\documentclass[a4paper,reqno]{amsart}

\usepackage{amssymb}
\usepackage{fullpage}
\usepackage{enumerate}
\usepackage{xspace}
\usepackage{amsthm,amsfonts,amsxtra,amscd}
\usepackage[all]{xy}
\usepackage{verbatim}

\usepackage{color}

\newcommand{\grado}    {\sfg\sfr}

\newcommand{\id}       {\mathrm{id}}
\newcommand{\tabr}     {\leftarrow}

\newcommand{\notcomparable} {not $\tabr$\,comparable\xspace}
\newcommand{\monomialOrder}{<_t}
\newcommand{\newMonomialOrder}{\prec_t}

\newcommand{\Rich}     {\calR}

\newcommand{\tom}      {{\tilde \omega}} 

\theoremstyle{plain}       
\newtheorem{lem}{Lemma}
\newtheorem{teo}[lem]{Theorem}

\newtheorem{cor}[lem]{Corollary}

\theoremstyle{definition}

\newtheorem{oss}[lem]{Remark}

\theoremstyle{remark}

\newcommand{\mA}{\mathbb A} 
 \newcommand{\mC}{\mathbb C}  
 \newcommand{\mF}{\mathbb F} \newcommand{\mG}{\mathbb G} 
   
  \newcommand{\mM}{\mathbb M} 
\newcommand{\mN}{\mathbb N}  \newcommand{\mP}{\mathbb P} 
  \newcommand{\mS}{\mathbb S}

\newcommand{\mZ}{\mathbb Z}

 \newcommand{\calF}{\mathcal F} 
  
 \newcommand{\calL}{\mathcal L} 
 \newcommand{\calO}{\mathcal O} 
 \newcommand{\calR}{\mathcal R} \newcommand{\calS}{\mathcal S}

\newcommand{\goK}{\mathfrak K}

\newcommand{\gob}{\mathfrak b}   
  \newcommand{\gog}{\mathfrak g}

\newcommand{\got}{\mathfrak t}

\newcommand{\sfA}{\mathsf A} 
\newcommand{\sfB}{\mathsf B} \newcommand{\sfC}{\mathsf C} \newcommand{\sfD}{\mathsf D}

  \newcommand{\sfS}{\mathsf S}

 \newcommand{\sff}{\mathsf f} \newcommand{\sfg}{\mathsf g}

 \newcommand{\sfr}{\mathsf r}

\newcommand{\gra}{\alpha} \newcommand{\grb}{\beta}       \newcommand{\grg}{\gamma}
\newcommand{\grd}{\delta} \newcommand{\gre}{\varepsilon} \newcommand{\grz}{\zeta}
 \newcommand{\grl}{\lambda}     
\newcommand{\grf}{\varphi}      

\newcommand{\grG}{\Gamma} \newcommand{\grD}{\Delta}
\newcommand{\grL}{\Lambda} 

\newcommand{\mi}  {\imath}

\newcommand{\mk}  {\Bbbk}

\newcommand{\lra}       {\longrightarrow}
\newcommand{\isocan}    {\simeq}

\newcommand{\cech}      {\spcheck}

\renewcommand{\geq}     {\geqslant}
\renewcommand{\leq}     {\leqslant}
\newcommand{\senza}     {\smallsetminus}
\newcommand{\ristretto} {\bigr|}
\newcommand{\hypOne}    {{\sf Hp1}}
\newcommand{\hypTwo}    {{\sf Hp2}}
\newcommand{\hypThree}  {{\sf Hp3}}
\newcommand{\hypFour}   {{\sf Hp4}}
\newcommand{\hypFive}   {{\sf Hp5}}
\newcommand{\hypSix}    {{\sf Hp6}}
\newcommand{\symOne}    {{\sf Sym1}}
\newcommand{\symTwo}    {{\sf Sym2}}
\newcommand{\modOne}    {{\sf Mod1}}
\newcommand{\modTwo}    {{\sf Mod2}}
\newcommand{\modThree}  {{\sf Mod3}}
\newcommand{\sphOne}    {{\sf Sph1}}
\newcommand{\sphTwo}    {{\sf Sph2}}
\newcommand{\sphThree}  {{\sf Sph3}}
\newcommand{\GL}        {{\sf GL}}
\newcommand{\SL}        {{\sf SL}}
\newcommand{\SO}        {{\sf SO}}
\newcommand{\Sp}        {{\sf Sp}}

            \newcommand{\st}       {\, : \,}
       
         \newcommand{\mand}     {\text{ and }}

\newcommand{\tableau}[2]{\begin{smallmatrix}#1\\#2\end{smallmatrix}}
\newcommand{\tableaun}[3]{\begin{smallmatrix}#1 & #2\\#3\end{smallmatrix}}

\definecolor{revisionColor}{rgb}{0.0,0.0,0.0}
\newcommand{\rev}[1]{{\color{revisionColor}{#1}}}

\DeclareMathOperator{\length}{length}

\title[Equations of model varieties]{Pl{\"u}cker relations and spherical varieties: application to model varieties}

\author{Rocco Chiriv\`\i\ and Andrea Maffei}



\address{Dipartimento di Matematica e Fisica ``Ennio De Giorgi''\\ Universit\`a del Salento\\ Via per Arnesano\\ 73047 Monteroni di Lecce (LE)\\ Italy}
\email{rocco.chirivi@unisalento.it}

\address{Dipartimento di Matematica\\ Universit\`a di Pisa\\ Largo Bruno Pontecorvo 5\\ 56127 Pisa (PI)\\ Italy}
\email{maffei@dm.unipi.it}
\begin{document}

\begin{abstract}
A general framework for the reduction of the equations defining classes of spherical varieties to (maybe infinite dimensional) grassmannians is proposed. This is applied to model varieties of type $\sfA$, $\sfB$ and $\sfC$; in particular a standard monomial theory for these varieties is presented.
\end{abstract}

\maketitle

\section{Introduction}\label{sec:introduction}

A standard monomial theory for an algebra $A$ over a field $\mk$ is given by a set of generators $\mathbb{A}$, together with a notion of standardness for monomials in $\mathbb{A}$ such that $A$ is spanned by standard monomials as a $\mk$--vector space; further the relations in $A$ writing non standard monomials in terms of standard ones, called straightening relations, are ``upper triangular'' (see Section \ref{sec:SMT} for details). One of the main purpose of standard monomial theory is to replace the knoweldge of the equations defining a variety by the order requirement in the straightening relations. Indeed\rev{,} in many situation\rev{s,} this weaker property is enough to prove geometric results like normality, Cohen-Maculay property, degeneration results and others. Moreover\rev{,} the order structure of the straightening relations allows to prove that they are generators for the ideal defining the algebra $A$ as a quotient of the symmetric algebra $\sfS(\mathbb{A})$. The first example of such a theory dates back to Hodge study in \cite{hodge} of the grassmannian of $k$--spaces in a $n$--dimensional vector space.

A standard monomial theory for flag and Schubert varieties has been developed over the years by Lakshmibai, Musili and Seshadri beginning with \cite{LMSBull}, this program culminated in the work  of Littelmann \cite{L:SMT} (see also \cite{chirivi}) where such a theory is defined in the generality of symmetrizable Kac-Moody groups.

At the same time, in \cite{DEP} a standard monomial theory for the coordinate ring of $\SL_n$ was reduced to that of the grassmannian of $n$--spaces in a $2n$--dimensional space. Next this result was generalized in various directions by many authors (see the introduction in \cite{CLM} for further details). In the paper \cite{CLM} the authors and Peter Littelmann shown how a standard monomial theory for certain classes of symmetric varieties may be described in terms of the Pl\"ucker relations of suitable, maybe infinite dimensional, grassmannians. Moreover all previous known cases of this type of reduction are particular instances of our construction for symmetric varieties.

The first purpose of the present paper is the development of a general framework for this reduction from the coordinate ring of a variety to the coordinate ring of a grassmannian. We propose how a suitable grassmannian for such process may be defined if we start with a spherical variety. However, this proposal does not work in general for all spherical varieties, indeed various technical hypotheses must be met. \rev{Anyway,} it is quite general and the hypotheses are fulfilled in many interesting cases. Let us explains our approach in more details.

Let $G$ be a semisimple and connected algebraic group and let $H$ be an algebraic subgroup such that $X \rev{\doteq} G/H$ is spherical. Let $\Lambda$ be the weight lattice and $\Lambda^+$ the monoid of dominant weights with respect to a fixed maximal torus and Borel subgroup of $G$. Denote by $\Omega^+$ the monoid of spherical weights, i.e. of dominant weights $\lambda$ such that the $H$--invariant subspace $V_\lambda^H$ of the $G$--irreducible module $V_\lambda$ of highest weight $\lambda$ is non\rev{-}zero. Our first hypotheses is that $\Omega^+$ is a free monoid; let $\gre_1,\gre_2,\ldots,\gre_\ell$ be generators for $\Omega^+$. Then the coordinate ring $A$ of $X$ is generated by $V_{\gre_1}^*, V_{\gre_2}^*,\ldots,V_{\gre_\ell}^*$. Our aim is to construct a standard monomial theory having as generators a basis of $V_{\gre_1}^*\oplus V_{\gre_2}^*\oplus\cdots\oplus V_{\gre_\ell}^*$.

The main request is the existence of a Kac-Moody group $K$ such that $G$ is the semisimple part of a Levi of $K$ with the following properties. There exists a suitable grassmannian $\mathcal{F}$ for $K$, a $G$--invariant Richardson variety $\mathcal{R}\subset\mathcal{F}$ and a line bundle $\calL$ on $\mathcal{F}$ such that $X$ may be embedded in a completion of $\mathcal{F}$, \rev{moreover} $H^0(\mathcal{R},\calL)\simeq V_{\gre_1}^*\oplus V_{\gre_2}^*\oplus\cdots\oplus V_{\gre_\ell}^*$ and $\oplus_{n\geq0} H^0(\mathcal{R},\calL^n)$ is isomorphic to the coordinate ring of \rev{$X$} as $G$--modules. In a way, this group $K$ is a bigger group of ``hidden'' symmetries for $X$.

Further we require the existence of an additive map $\grado:\Omega^+\longrightarrow\mN$ such that the following compatibility with tensor product of spherical modules holds: for all $\mu,\lambda,\nu\in\Omega^+$ such that $V_\nu^*\subset V_\lambda^*\otimes V_\mu^*$ we have $\grado(\nu)\leq\grado(\lambda+\mu)$. We require also that the generators $\gre_1,\gre_2,\ldots,\gre_\ell$ are linearly ordered by $\grado$. Finally a certain compatibility between the function $\grado$ and the multiplication of sections in $\Gamma(\mathcal{F},\mathcal{L})$ is required. 

Once all such hypotheses are fulfilled we are able to prove that the relations among a basis of $V_{\gre_1}^*\oplus V_{\gre_2}^*\oplus\cdots\oplus V_{\gre_\ell}^*$ may be described in terms of the Pl\"ucker relations of $\mathcal{F}$, and for this reason we call the above general framework \emph{pl\"uckerization} for $X$. Further a standard monomial theory for $X$ may be described in terms of the standard monomial theory of $\mathcal{R}$, see Theorem \ref{teo:generalResult}. Using this we see that $X$ degenerates to $\mathcal{R}$ in a $G$--equivariant flat way, see Corollary \ref{cor:generalDegeneration}.

The construction of $K$, $\calF$, $\calR$, $\grado$,... follows an empirical recipe. The main ingredients for this construction are suggested by the moltiplication rule of the modules $V_{\gre_1}^*,V_{\gre_2}^*,\ldots,V_{\gre_\ell}^*$: see \symOne,\ \symTwo,\ \modOne,\ \modTwo,\ \modThree\ in Section \ref{sec:modello} and \sphOne,\ \sphTwo,\ \sphThree\ in Section \ref{sec:another}. Once these objects are defined the verifications of the above technical hypotheses are very uniform for the different varieties in the applications. In particular this recipe hints how many nodes to add to the Dynkin diagram of $G$ in order to obtain $K$; for symmetric varieties just one node while for model varieties and another class of spherical varieties (see below) two nodes are needed.

Our previous paper \cite{CLM} with Peter Littelmann follows the above general framework applying it to certain classes of symmetric varieties. Notice\rev{,} however\rev{,} that in that paper the proof of the existence of a standard monomial theory derived by that of the bigger group $K$ is wrong; there we tacitly assumed that a certain map is $G$--equivariant \rev{while} this is not the case in general (see \rev{R}emark \ref{remark_controesempio} below). However\rev{,} Theorem \ref{teo:generalResult} of this paper amends that gap.

The second aim of \rev{the present} paper is the application of the above described framework to model varieties of type $\sfA$, $\sfB$ and $\sfC$. A homogeneous model variety for a semisimple group $G$ is a homogeneous quasi--affine variety whose regular function ring is the sum of all irreducible representations of $G$ with multiplicity one. These varieties were introduced by Bernstein, Gelfand and Gelfand in \cite{Gel3} (see also Luna's paper \cite{Luna}) and studied by Gelfand and Zelevinsky in \cite{Gel2} and \cite{Gel1}. In particular for a homogeneous model variety $G/H$ we have $\Omega^+=\Lambda^+$.

In the cited papers the authors provided an embedding of the model varieties for classical groups as an open subset of a grassmannian of a bigger finite dimensional group; hence there are some similarities with our program. From the geometrical viewpoint\rev{,} the construction of Gelfand and Zelevinsky is more natural than our approach. However\rev{,} their embedding is not suitable for the application to the standard monomial theory having as generators a basis of $V_{\gre_1}^*\oplus V_{\gre_2}^*\oplus\cdots\oplus V_{\gre_\ell}^*$. Indeed\rev{,} it is for this purpose that we need to use a more complicated infinite dimensional grassmannian for model varieties of type $\sfB$ and $\sfC$. The two approaches coincide for the model variety of type $\sfA$ for which we use a finite dimensional lagrangian grassmannian.

Finally we present a further application of our framework to another class of spherical varieties. For this example the recipe for the construction of $K$ is a bit different of the above reported one; we have included this class of varieties as an illustration of how our program may be applied in other cases.

The paper is organized as follows. In Section \ref{sec:simboli} we introduce some general notations and symbols we use in the other sections. In Section \ref{sec:SMT} we review the standard monomial theory structure for flag, Schubert and Richardson varieties. Moreover we introduce the class of ridge Richardson varieties which plays a fundamental role in our theory. In Section \ref{sec:pluckerization} we present a general framework for the reduction of the equations of certain varieties to those of a suitable grassmannian. In the next Section \ref{sec:modello} we apply this construction to the case of model varieties of type $\sfA$, $\sfB$ and $\sfC$. Finally, as an example of possible further applications, we see in Section \ref{sec:another} how our theory may be applied to another class of spherical varieties.

\vskip 0.5cm

\noindent{\bf Acknowledgements.} We would like to thank Paolo Bravi for useful conversations.

\section{General assumptions, conventions and notations}\label{sec:simboli}

All groups, varieties and ind-varieties are over a fixed algebraically closed field $\mk$ of characteristic zero.

We denote by $K$ a Kac-Moody group constructed as in Kumar's book \cite{Kumar} ch. VI, and we denote by $\goK$ its Lie algebra; we always assume that $K$ is symmetrizable. Further we use the following notations and make the following assumptions:

\begin{itemize}
\item $T_K\subset B_K\subset K$ are a maximal torus and a Borel subgroup of $K$, respectively, and $\got_K\subset \gob_K\subset\goK$ are their Lie algebras;
\item $\Lambda_K$ is the character lattice of $T_K$ and $\Phi_K\subset\Lambda_{\rev{K}}$ is the set of roots;
\item $\Delta_K\subset \Phi_K$ is the set of simple roots determined by $B_K$ and for each $\gra\in \Delta_K$ we denote by $\gra\cech\in \got_K$ the corresponding coroot;
\item \rev{we denote by $<$ the dominant order on $\Lambda_K$;}
\item we assume that there exists a set of fundamental weights in $\Lambda_K$: that is a set $\{\tom_\gra:\gra\in\Delta_K\}$ such that $\langle\gra\cech;\tom_\grb\rangle=\grd_{\gra\grb}$ for all $\gra,\grb \in \Delta_K$ (in Kumar's construction in \cite{Kumar} such a set always exists);
\item for all $\grl,\mu\in \Lambda_K$ we write $\grl\equiv\mu$ if $\langle\gra\cech;\grl\rangle = \langle\gra\cech;\mu\rangle$ for all $\gra\in\Delta_K$;
\item $W_K$ denotes the Weyl group of $K$ w.r.t.\ $T_K$, and, for a real root $\alpha\in\Phi_K$, we denote by $s_\alpha$ the reflection defined by $\alpha$;
\item for a standard Levi $L$ relative to $T_K, B_K$, we set $T_L\doteq T_K$, $B_L\doteq B_K\cap L$, hence $\Delta_L\subset\Delta_K$, $W_L\subset W_K$; further $W^L\subset W_K$ denotes the set of minimal representatives of $W_K/W_L$;
\item for a parabolic $P\supset B_K$, we define $W_P\doteq W_L$, $W^P\doteq W^L$ where $L$ is the standard Levi of $P$,
\item for $u\in W_K$, $\length(u)$ is the length of $u$ with respect to $\Delta_K$.
\end{itemize}

\section{Standard monomial theory}\label{sec:SMT}

We recall the definition of a standard monomial theory. 

Let $A$ be a commutative $\mk$--algebra\rev{,} $\mA$ a finite subset of $A$ and let $\tabr$ be a transitive antisymmetric binary relation ({\it t.a.b.r.} for short) on $\mA$. (Note, $\tabr$ is not necessarily reflexive.)  If $a_1 \tabr a_2 \tabr \cdots \tabr a_n$, then we say that the monomial $a_1 a_2 \cdots a_n$ is a \emph{standard monomial}. We denote by $\mS\mM(\mA)$ the set of all standard monomials. We say that $(\mA, \tabr)$ is a standard monomial theory (for short SMT) for $A$ if $\mS\mM(\mA)$ is a $\mk$--basis of $A$.

The construction of a standard monomial theory often comes together with the description of the straightening relations, i.e. a set of relations in the elements of $\mA$ which provide an inductive procedure to rewrite a non standard monomial as a linear combination of standard monomials. Let us explain this in more details.

Let $(\mA, \tabr)$ be a SMT for the ring $A$. In particular, $\mA$ generates $A$ and we denote by $Rel_A$ the kernel of the natural morphism from the symmetric algebra $\sfS(\mA)$ to $A$. Let $\mM(\mA)\subset\sfS(\mA)$ be the set of all monomials in the set of generators $\mA$ and let $\monomialOrder$ be a monomial order which refines the t.a.b.r. on $\mA$. (We recall that a monomial order is a total order on the set of monomials such that: (i) if $m,m',m''$ are monomials and $m'\monomialOrder m''$ then $mm'\monomialOrder mm''$ and (ii) $1\monomialOrder m$ for all monomials $m\neq 1$ (see \cite{Eisenbud}, section 15.2).) Let us assume, for any two $a, a' \in \mA$ which are \notcomparable, that there exists a relation $R_{a,a'} \in Rel_A$ such that
$$
R_{a,a'}=a\,a' - P_{a,a'}
$$
and $P_{a,a'}$ is a sum of monomials which are strictly smaller than $a\,a'$ with respect to the order $\monomialOrder$. A set of relations satisfying these properties is called a set of \emph{straightening relations}. We have the following simple lemma (see \cite{CLM} Lemma 10).

\begin{lem}\label{lem:SMT}
Let $(\mA, \tabr)$ be a SMT for the ring $A$ and let $\calR \rev{\doteq} \{R_{a,a'}\st a,a' \in \mA$ are \notcomparable$\}$ be a set of straightening relations. Then $\calR$ generates $Rel_A$.
\end{lem}

\subsection{Standard monomial theory for flag and Schubert varieties} \label{ssec:SMTflag}

Let $A$ be the coordinate ring of the cone over a generalized flag variety $\calF$ of a symmetrizable Kac-Moody group $K$.  For this type of algebras a standard monomial theory has been constructed in \cite{L:SMT} and \cite{chirivi}. We recall the main properties of this SMT. 

\rev{We fix} an ample line bundle $\calL$ over $\calF$ and consider the ring $\rev{\grG_\calF}\rev{\doteq}\bigoplus _{n\geq 0} \grG(\calF,\calL^{n})$.  A basis $\rev{\mF}$ of $\rev{\grG_\calF}$ has been constructed in \cite{L:SMT} together with a t.a.b.r. $\tabr$ on this set such that $(\rev{\mF}, \tabr)$ is a SMT for $\rev{\grG_\calF}$.

We denote by $\rev{\mS\mM(n)}$ the set of standard monomials of degree $n$, by $\rev{\mS\mM}$ the set of all standard monomials and by $\rev{\mM}$ the set of all monomials in the set of generators $\rev{\mF}$.  For $f,f' \in\rev{\mF}$ which are \notcomparable, the product $f\,f'$ can be expressed as a sum $P_{f,f'}$ of standard monomials of degree two. In \cite{chirivi} a total order $\monomialOrder$ has been introduced on $\rev{\mM}$ with the properties required in the previous discussion of a general SMT, so that the relations $R_{f,f'}=f\,f' -P_{f,f'}$ form a set of straightening relations. These relations are called Pl\"ucker relations since they generalize the usual Pl\"ucker relations for the Grassmannian.

Furthermore, this theory is adapted to Schubert varieties. Indeed let $\calS\subset\calF$ be a Schubert variety, i.e. a closed $B_K$--stable subvariety, and set $\rev{\grG_\calS} \rev{\doteq}\bigoplus_{n\geq 0} \grG(\calS,\calL^n)$. Denote by $r :\rev{\grG_\calF} \lra \rev{\grG_\calS}$ the restriction map, let $I_\calS$ be its kernel and define $\rev{\mF_\calS} \doteq \{a \in \rev{\mF} \st r(a)\neq 0\}$. Then the set $\{r(a) \st a \in \rev{\mF_\calS}\}$ with the t.a.b.r.\ induced by the t.a.b.r.\ $\tabr$ of $\mF_\calS$ realizes a SMT for $\rev{\grG_\calS}$ and the monomials $m \in \rev{\mS\mM}$ which contain elements not in $\rev{\mF_\calS}$ form a $\mk$--basis of $I_\calS$. We define also a ``restriction'' map, \rev{which} we also denote with the symbol $r:\sfS(\mF)\lra \sfS(\mF_\calS)$ by sending to zero all the elements of $\mF \senza \mF_\calS$ and we notice that the restrictions $r(R_{f,f'}) \in \sfS^2(\mF_\calS)$ of the relations $R_{f,f'}$ {to $\calS$}, for $f,f' \in \mF_\calS$ which are \notcomparable, form a set of straightening relations. Summarizing we have:
\begin{teo}[\cite{L:SMT}]\label{thm:SMTFlagSchubert}\hfill
\begin{itemize}
\item[i)] $(\mF,\tabr)$ is a SMT for $\grG_\calF$, and the relations $R_{f,f'}$ for $f,f' \in \mF$ which are \notcomparable form a set of straightening relations.
\item[ii)] $(\{r(a)\mid a\in\mF_\calS\},\tabr)$ is a SMT for $\grG_\calS$, and the relations ${r}(R_{f,f'}) \in \sfS^2 (\mF_\calS)$ for $f,f' \in \mF_\calS$ which are \notcomparable form a set of straightening relations. Moreover, the kernel $I_\calS$ of the restriction map has a $\mk$--basis consisting of the set of all standard monomials which contain elements not in $\mF_\calS$.
\end{itemize}
\end{teo}
The elements of $\mF$ are eigenvectors for the action of $T_K$ and we denote by $weight(f)$ the weight of $f\in \mF$ w.r.t.\ the action of $T_K$. The t.a.b.r.\ $\tabr$ is compatible with the dominant order in the following sense: if $f\tabr f'$ and $f\neq f'$ then $weight(f)< weight(f')$ w.r.t.\ the dominant order. Moreover $\mF$ has a minimum $f_0$ which is a lowest weight vector.

A Richardson variety is the closure of the intersection of a $B_K$--orbit with an orbit of the Borel subgroup opposite to $B_K$. In this paper we are interested in a particular type of Richardson variety, namely given a Schubert variety $\calS$ we will consider the Richardson variety $\calS_0 \rev{\doteq}\{y\in \calS\st f_0(y)=0\}$. The above described SMT and Theorem \ref{thm:SMTFlagSchubert} for $\calS$ immediately generalize to $\calS_0$ by choosing as set of generators $\mF_0\rev{\doteq}\mF_\calS\senza\{f_0\}$. \rev{Further we denote by $\mM_0$ the set of standard monomials not containing $f_0$ and by $\mS\mM_0$ the subset of those standard monomials not containing $f_0$.}

\subsection{Ridge Richardson varieties and compatibility with Levi factors} \label{ssec:compatibilita}

Let $L$ be a finite type standard Levi of $K$ and let \rev{$\calS$} be an $L$--stable \rev{Schubert} variety. The inclusion $\rev{\mF_\calS}\subset\rev{\mF}$ gives a vector space \rev{injection} $\Gamma(\rev{\calS},\calL)\longrightarrow\Gamma(\calF,\calL)$; this \rev{map} is not $L$--equivariant \rev{in} general.

\begin{oss}\label{remark_controesempio} Although the example is very simple, we need to introduce some details.

Let $K=\mathsf{SL}(3,\mC)$, let $L$ be the set of matrices stabilizing the decomposition $\mC^3=\mC^2\oplus\mC$ and let $\mathcal{F}$ be the full flag variety. We choose as Borel $B$ the set of upper triangular matrices in $G$.

We take the line bundle $\mathcal{L}$ realizing the embedding of the cone over $\mathcal{F}$ in the vector space $V$ of $3\times 3$ traceless complex matrices (i.e. in the Lie algebra of $G$). Let $M\doteq (x_{i,j})_{1\leq i,j\leq 3}$ be the matrix of coordinates of the space of $3\times 3$ matrices. Given a sequence $R=i_1i_2\cdots i_r$, with $r\leq3$, let $d(R)$ be the function on $V$ given by the determinant of the submatrix of $M$ with rows $i_1,i_2,\ldots,i_r$ and coloumns $1,2,\ldots,r$. As computed in \cite{chThesis}, for this embedding the Littelmann basis $\mF$ of $\Gamma(\mathcal{F},\mathcal{L})$ is given by the functions
$$
p(\tableau{R_1}{R_2})\doteq d(R_1)d(R_2)
$$
with $\tableau{R_1}{R_2}$ one the following tableaux
$$
\tableaun{1}{2}{1},\ \tableaun{1}{2}{2},\ \tableaun{1}{3}{1},\ \tableaun{1}{2}{3},\ \tableaun{2}{3}{1},\ \tableaun{1}{3}{3},\ \tableaun{2}{3}{2},\ \tableaun{2}{3}{3}.
$$
The Schubert varieties $X_\tau$ in $\mathcal{F}$ are indexed by permutations $\tau\in S_3$; we consider the Schubert variety $\mathcal{S}\doteq X_{(123)}$. The restriction of the functions $p(\tableau{R_1}{R_2})$ not vanishing on $\mathcal{S}$ are a basis $r(\mF_{\mathcal{S}})$ of $\Gamma(\mathcal{S},\mathcal{L})$; in our situation $\mF_{\mathcal{S}}$ is $p(\tableaun{1}{2}{1})$, $p(\tableaun{1}{2}{2})$, $p(\tableaun{1}{3}{1}),\,p(\tableaun{2}{3}{1}),\,p(\tableaun{2}{3}{2})$.

An easy computation shows that the image of the inclusion of $\Gamma(\mathcal{S},\mathcal{L})$ in $\Gamma(\mathcal{F},\mathcal{L})$ induced by $\mF_{\mathcal{S}}\subset\mF$ is not stable \rev{under} the action of the Lie algebra of $L$\rev{;} hence \rev{this inclusion} is not $L$--equivariant.
\end{oss}

Let $\grz$ be the lowest weight of $\Gamma(\calF,\calL)$. A $L$--stable Richardson variety $\calR\subset \calF$ is said \emph{extremal} (w.r.t.\ $\calL$) if the lowest weights of the irreducible $L$--submodules of $\Gamma(\calR,\calL)$ are in the $W_K$--orbit of $\grz$.

Now we define a class of $L$--stable extremal Richardson varieties for which we can describe the coordinate ring. Let $Q$ be the standard parabolic such that $\calF=K/Q$, let $w_0,w_1,\cdots,w_\ell\in W^Q$ be a sequence of minimal representatives with the following properties:
\begin{itemize}
\item $w_0=e$, $w_h=s_{\gamma_h}u_hw_{h-1}$ with $\gamma_h\in\Delta_K\setminus\Delta_L$ and $u_h\in W_L$ for $h=1,2,\ldots,\ell$;
\item $\length(w_h)=1+\length(u_h)+\length(w_{h-1})$ for $h=1,2,\ldots,\ell$;
\item setting $\grz_h\doteq -w_h(\grz)$, we have $\langle\gamma_h\cech,\grz_h\rangle=-1$ and $\langle\gamma\cech,\grz_h\rangle=0$ for all $\gamma\in\Delta_K\setminus\Delta_L$ and $\gamma\neq\gamma_h$,
\item $u_k(\grz_h)=\grz_h$ for $0\leq k\leq h$.
\end{itemize}

Define $\tau$ as the minimal representative of $w_Lw_\ell$, where $w_L$ is the longest element of $W_L$; let $\calS\doteq\overline {B_K\tau Q/Q}$ be the Schubert variety of $\calF$ corresponding to $\tau$. The following theorem is proved as Theorem 39 in \cite{CLM}.

\begin{teo}
With the above introduced notations we have: the Schubert variety $\calS$ is extremal, $L$--stable and, moreover, for all $n\geq 0$ we have
$$
\Gamma(\calS,\calL^n)=\bigoplus_{0\leq i_1\leq i_2\leq\cdots\leq i_n\leq \ell}V_{\grz_{i_1}+\grz_{i_2}+\cdots+\grz_{i_n}}^*
$$
as $L$--modules.
\end{teo}

If $L\subset Q$ then the lowest weight vector $f_0$ of $\Gamma(\calS,\calL)$ is an $L$--eigenvector (of weight $\grz$) and the Richardson variety $\rev{\calR}\doteq\calS_0$ is $L$--stable. So we have the following corollary (see Corollary 40 in \cite{CLM}).

\begin{cor}\label{cor:decR}
The Richardson variety \rev{$\calR$ is $L$--stable and extremal} and\rev{,} for all $n\geq0$,
$$
\Gamma(\rev{\calR},\calL^n)=\bigoplus_{1\leq i_1\leq i_2\leq\cdots\leq i_n\leq \ell}V_{\grz_{i_1}+\grz_{i_2}+\cdots+\grz_{i_n}}^*
$$
as $L$--modules.
\end{cor}

We call $\rev{\calR}$ the \emph{ridge} Richardson variety defined by $w_0,\dots,w_\ell$ \rev{and we let $s:\Gamma(\calR,\calL)\longrightarrow\Gamma(\calF,\calL)$ be the corresponding inclusion.}

\section{Pl\"uckerization}\label{sec:pluckerization}

In this section we propose a general pattern for determining a SMT and the straightening relations of the coordinate ring of a spherical homogeneous variety. This method does not always work, nevertheless it works in some interesting cases we present in the next sections. We call this method \emph{pl\"uckerization} since it reduces the computation of the relations of a ring to the Pl\"ucker relations for a generalized flag variety. The main ingredient is the introduction of a bigger group of symmetries into the problem.

In what follows we denote by $G$ a fixed semisimple, connected and simply connected algebraic group, $\gog$ its Lie algebra and $G\lra \bar G$ an isogeny.  We denote all the objects defined in Section \ref{sec:simboli} $\got_G,\gob_G$, etc.\ simply by $\got, \gob$, etc. without the subscript $G$. The fundamental weights of $\gog$ will be denoted by $\omega_\gra$, $\alpha\in\Delta$.

Let $H$ be an algebraic subgroup of $G$, $X \rev{\doteq} G/H$ and $A$ the ring of regular functions $\mk[X]$.  We assume that $X$ is spherical, so every irreducible representation appears in $A$ with multiplicity at most one. Let $\Omega_A^+ \rev{\doteq}\{\grl \in \grL^+: V_\grl^H \neq 0\}$ so that $A \isocan \bigoplus_{\grl\in \Omega_A^+}V_\grl^*$ as a $G$\rev{--}module. If $M$ and $N$ are vector subspaces of $A$ we denote by $M\cdot N$ the vector space spanned by the products $m\cdot n$ with $m\in M$ and $n\in N$ in the ring $A$.

In order to develop our method we need to assume various hypotheses on $A$ that we denote by $\hypOne$, $\hypTwo$,... \rev{In stating each hypothesis we assume the preceding ones.}

First of all notice that $\Omega_A^+$ is a submonoid of $\grL$; we assume

\begin{itemize}
\item[\hypOne:] $\Omega_A^+$ is a free monoid.
\end{itemize}

We denote by $\gre_1, \gre_2, ... , \gre_\ell$ a basis of $\Omega_A^+$ and we define $A_j$ as the sum of all submodules $V_\grl^*$ of $A$ with $\grl = \sum_{i}a_i \gre_i$ and $\sum_{\rev{i}} a_i = j$; further $A'_j\rev{\doteq} \bigoplus_{i\leq j} A_i$. Notice that, $X$ being irreducible, the product $V_\grl^*\cdot V_\mu^*$ in the ring $A$ contains $V_{\grl+\mu}^*$. Hence $A'_i \cdot A'_j = A'_{i+j}$, and $A$ is generated by $A_1$ as a $\mk$--algebra. Our aim is to describe the ring $A$ with respect to this natural choice of generators. We denote by $Rel_A$ the kernel of the natural map $\psi:\sfS(A_1)\lra A$.

The heart of our assumption is the existence of a Richardson variety whose coordinate ring is isomorphic to $A$ as a $G$--module. Let $K$ be a Kac-Moody group which contains $G$ as the semisimple part of a Levi factor $L$. Let $B_K$ be a Borel subgroup of $K$ such that $B_K \cap G = B$, $T_K \subset B_K$ a maximal torus of $K$ containing $T$, $\grL_K$ the set of weights of $T_K$, $\grD_K\supset \grD$ the set of simple roots of $K$ w.r.t.\xspace these choices and $\grD_0 \rev{\doteq} \grD_K \senza \grD$.  Let $Q$ be a parabolic of $K$ containing $G$, $\calF \rev{\doteq} K / Q$ the associated generalized flag variety, $\calL$ an ample equivariant line bundle on $\calF$ and $\calR \subset \calF$ a $G$--stable ridge Richardson variety w.r.t.\xspace $\calL$ defined by $w_0,\,w_1,\ldots,w_\ell$ as in the previous section. Our main hypotheses are the following:

\begin{itemize}
\item[\hypTwo:] for all $i$ we have $A_i\isocan\grG(\calR,\calL^i)$ as $G$--modules, in particular $A\isocan\grG_\calR$.
\item[\hypThree:] there exists a $G$--equivariant morphism of $\mk$--algebras $\grf : \grG_\calF \lra A$. We define $\grf_i \rev{\doteq} \grf\ristretto_{\grG(\calF,\calL^i)}$.
\end{itemize}

Now let $1\leq h\leq\ell$ and let $V_h$ be the $G$--submodule of $\Gamma(\mathcal{F},\mathcal{L})$ generated by the line of vectors of weight $w_h(\grz)$; moreover let $V \rev{\doteq} V_1\oplus V_2\oplus\cdots\oplus V_\ell$. By Corollary \ref{cor:decR} we have $V\isocan\Gamma(\calR,\calL)$, hence by \hypTwo\ in the case $i=1$, setting $\grz_h\rev{\doteq}-w_h(\grz)$ as in the previous section and up to renumbering $\gre_1,\gre_2,\ldots,\gre_\ell$, we have $\grz_{h|T}=\gre_h$ for $h=1,2,\ldots,\ell$. In particular $V_h\simeq V_{\gre_h}^*$ for $h=1,2,\ldots,\ell$.

In order to relate the ring $\grG_\calR$ and $A$ we need also the following strenghtening of \hypThree\ in the case $i=1$

\begin{itemize}
\item[\hypFour:] The map $\varphi_1$ is an isomor\rev{ph}ism of $G$--modules from $V$ to $A_1$.
\end{itemize}

We now introduce an order on $G$--modules and some compatibility conditions with a $G$--stable decomposition of $\grG(\calF,\calL)$.  The space $Z^* \rev{\doteq} \grG(\calF,\calL)$ is an integrable lowest weight irreducible representation of $K$ of lowest weight $\zeta$; let $f_0 \in \mF$ be a lowest weight vector in $Z^*$. Let $C$ be the identity component of the subgroup of elements of $T_K$ which commute with $G$. 

We can decompose $Z^*$ into \rev{isotypical} components w.r.t. the action of $C$, these are described as follows: if $\eta \in \mN[\grD_0]$ we define $Z^*_\eta$ as the subspace generated by all $T_K$--weight vectors whose weight is equal to $\zeta+\eta+\grg$ with $\grg \in \mN[\grD]$.  These subspaces are finite dimensional and $G$--stable.  If $\eta = \sum_{\gra\in\Delta_0} a_\gra \gra$ we define $ht_0(\eta) = \sum a_\gra$ and we set $Z^*_n \rev{\doteq} \bigoplus Z^*_\eta$ where the sum is over all the elements $\eta\in \mN[\Delta_0]$ such that $ht_0(\eta)=n$. Notice that for $h=1,2,\ldots,\ell$, by the definition of $\grz_h$, $V_h$ is a submodule of $Z_h^*$. Further the $\mk$--span of $\mF_\calR$ is a subvector space of $Z_1^*\oplus Z_2^*\oplus\cdots\oplus Z_\ell^*$ by construction of $\calR$.

Let $\grado:\Omega_A^+\longrightarrow\mN$ be the additive extension of $\grado(\gre_h)=h$ for $h=1,2,\ldots,\ell$.

\begin{itemize}
\item[\hypFive:] For all $\grl, \mu, \nu \in \Omega_A^+$ if $V_\nu^*\subset V^*_\grl \otimes V^*_\mu$ then $\grado(\nu) \leq\grado(\grl)+\grado(\mu)$.
\item[\hypSix:] For all $n\in\mN$ if $M$ is a $G$--irreducible submodule of $Z^*_n$ of highest weight $\grl\in \Omega_A^+$ then $\grado(\grl) \leq n$. Further, equality holds if and only if either $1\leq n\leq \ell$ and $M=V_n$ or $n=0$ and $M=\mk\,f_0$.
\end{itemize}

Notice that a combinatorial condition which implies \hypFive\ is the following: for all $\grl, \mu, \nu \in \Omega_A^+$ if $\nu \leq \grl + \mu$ then $\grado(\nu) \leq \grado(\grl)+\grado(\mu)$.

If $M$ is a $G$--module and $n \in \mN$ we define $M_{\grado \leq n}$ as the sum of all \rev{isotypical} components of type $V_\grl^*$ with $\grl \in \Omega_A^+$ and $\grado(\grl)\leq n$. Under \hypFive\ we have that $A_{\grado \leq n} \cdot A_{\grado\leq m} \subset A_{\grado \leq n+m}$ and if $n, m \geq \ell$ then $A_{\grado \leq n} \cdot A_{\grado\leq m} = A_{\grado \leq n+m}$. Since $\calR$ is irreducible, the same properties hold for $(\grG_\calR)_{\grado \leq n}$ too.

Let $r:\Gamma(\calF,\calL)\longrightarrow\Gamma(\calR,\calL)$ be the restriction map as in the previous sections. Since, as noted above, the two $G$--modules $\grG(\calR,\calL)$ and $V$ are isomorphic, we have a partition $\mF_0 = \coprod_{i=1}^\ell \mF_0(\gre_i)$; where $r(\mF_0(\gre_i))$ is a basis of a $G$--submodule of $\grG(\calR,\calL)$ isomorphic to $V_{\gre_i}^*$.

We are now in a position to introduce the basis we will use to express our standard monomial theory for the ring $A$. For $m\in\mM$ let $g_m\doteq\varphi(m)$, then we have

\begin{lem}\label{lemma:G0basis}
The set $\mG_0\doteq\{g_f \st f\in\mF_0\}$ is a basis of $A_1$ as a $\mk$--vector space.
\end{lem}
\begin{proof} Since $A_1$ is isomorphic to $V$ which, in turn, is isomorphic to $\grG(\calR,\calL)$, we find that $\mG_0$ has cardinal\rev{i}ty $\dim A_1$; hence we have only to show that the elements $g_f$, for $f\in\mF_0$, are linearly independent. So let $v$ be an element of the kernel of $\varphi$ restricted to the $\mk$--vector space spanned by $\mF_0$.

For $h=1,2,\ldots,\ell$ denote by $\rev{U}_h$ the sums of all $G$--submodules of $Z_h^*$ but $V_h$. Notice that the kernel of the restriction map $r:\Gamma(\calF,\calL)\longrightarrow\Gamma(\calR,\calL)$ contains $Z_0^*$ and $Z_h^*$ for all $h>\ell$, while $\ker r_{|Z_h^*}=\rev{U}_h$ for $h=1,2,\ldots,\ell$.

We proceed by contradiction; let us write $v=\sum_{h=1}^n v_h+\rev{u}_h$ with $v_h\in V_h$, $\rev{u}_h\in \rev{U}_h$ and $n$ maximal such that $v_n+\rev{u}_n\neq0$. Notice that $\varphi(v_n)$ is an element of the $G$--submodule of $A_1$ isomorphic to $V_n$ by \hypThree. For any $\lambda\in\Omega_A^+$ such that $V_\lambda^*$ is a $G$--submodule of $\rev{u}_n$ or of $Z_h^*=V_h+\rev{u}_h$ \rev{with $h<n$}, we have $\grado(\gre_n)=n>\grado(\lambda)$ by \hypSix. This shows that\rev{, in the above sum expressing $v$,} $v_n$ is the unique vector \rev{belonging} to the dual of a module of weight $\gre_n$. Hence the same is true for $\varphi(v_n)$, since $\varphi$ is a map of $G$--modules; so $\varphi(v_n)=0$. We conclude that $v_n=0$ since $\varphi$ is an isomorphism from $V$ to $A_1$.

Now we have $r(v_n+\rev{u}_n)=r(\rev{u}_n)=0$. But the restriction map $r$ is an injection of the $\mk$--span of $\mF_0(\gre_n)$ in $\Gamma(\calR,\calL)$, hence we find $v_n+\rev{u}_n=0$. This finishes the proof since we assumed that $v_n+\rev{u}_n\neq0$.
\end{proof}
On $\mG_0$ we define the same t.a.b.r.\xspace of $\mF$; that is\rev{:} $g_f \tabr g_{f'}$ if and only if $f \tabr f'$. In order to avoid possible confusion between expressions in $\sfS(\mG_0)$ and elements in $A$, we introduce the ring $\mk[u]\doteq\mk[u_f\st f\in \mF_0]$. If $m=f_1\cdots f_s\in\mM_0$ let $u_m\doteq u_{f_1}\cdots u_{f_s}$ be the corresponding monomial in $\mk[u]$. Since by the previous \rev{L}emma $\mG_0$ is a basis of $A_1$, \rev{$\mk[u]\simeq\sfS(A_1)$ via the map $\psi:\mk[u]\longrightarrow A$ given by} $\psi(u_f)\doteq g_f$.

We introduce on $\mk[u]$ a degree according to the map $\grado$: for $f \in \mF_0 (\gre_i)$ set the degree of $u_f$ equal to $\grado(\gre_i)=i$ and denote by $\grado(r)\in \mN$ the degree of an element $r$ in $\mk[u]$. If $m,m' \in \mM_0$, then we define $u_m\newMonomialOrder u_{m'}$ if either $\grado(u_m) < \grado(u_{m'})$ or $\grado(u_m) = \grado(u_{m'})$ and $m \monomialOrder m'$. Notice that the order $\newMonomialOrder$ is compatible with the \rev{relation} $\tabr$ on $\mG_0$.

We are now ready to introduce the straightening relations. If $f_1,f_2 \in \mF_0$ are not $\tabr$ comparable let $R_{f_1,f_2}\rev{\doteq}f_1f_2-P_{f_1,f_2} \in \sfS^2(Z^*)$ be the straightening relation for $\grG_{\mathcal{F}}$. Applying $\grf$ to this relation we have a relation on the ring $A$. However\rev{,} the polynomial $P_{f_1,f_2}$ is not just a polynomial in the variables $\mF_0$ but also in other variables in $\mF$. Let $\mF_2$ be the subset of $\mF$ which span $\bigoplus_{\rev{n}} Z^*_n$ with $n \leq 2\ell$. Let $\mF_3 \doteq \mF_2\senza \mF_0$. Since the relations are $T_K$--homogeneous (hence $C$--homogeneous) the polynomial $P_{f_1,f_2}$ can be written as $P_{f_1,f_2}^0+P_{f_1,f_2}^1$ where $P^0$ is the sum of all the monomials whose variables are in $\mF_0$ and $P^1$ is a sum of monomials whose variables are in $\mF_3$ or mixed in $\mF_0$ and $\mF_3$. For some applications (mainly to describe an equivariant deformation of $A$ to $\grG_\calR$ ), it will be convenient to change the basis $\mF_3$.

The span of $\mF_3$ is $G$--stable but the basis is not compatible with the decomposition of $Z^*$ in $G$--submodules. Let $\mF_1$ be a basis of $T_K$\rev{--eigenvectors} which span the same subspace as $\mF_3$ but \rev{that} is compatible with the decomposition of $Z^*$ into $G$--submodules. So we can consider $P^1$ also as a combination of monomials in the variables $\mF_1$ or mixed in $\mF_1$ and $\mF_0$. For each $f \in \mF_1$ let $\grl_f \in \grL⁺$ \rev{be} such that the $G$--module generated by $f$ is isomorphic to $V_{\grl_f}^*$. If $\grl_f \notin \Omega^+_A$ then $\grf(f)=0$ and, on the other hand, if $\grl_f = \sum_ia_i \gre_i$ then the multiplication $V_{\gre_1}^{*a_1}\cdots V_{\gre_\ell}^{*a_\ell}$ contains the module $V_{\lambda_f}^*$, hence there exists a polynomial $F_f(u) \in \mk[u]$ with $\grado(F_f(u))=\grado(\grl_f)$ such that $\psi(F_f) = \grf(f)$.

\begin{oss}\label{oss:calcolo}
Notice that, once we know $\grf(f)$, the computation of the polynomial $F_f$ does not depend on the knowledge of the multiplication in the ring $A$ but only on the projection $\sfS^{a_1}(V^*_{\gre_1})\otimes \cdots \otimes \sfS^{a_\ell}(V^*_{\gre_\ell}) \lra V_{\grl_f}^*$ which, up to a non\rev{-}zero scalar, is uniquely determined by the requirement to be $G$--equivariant. \rev{Indeed, $V_{\lambda_f}^*$ appears with multiplicity one in the tensor product being $\lambda_f$ its highest weight vector.}
\end{oss}

Now we define polynomials $\hat P ^0_{f_1,f_2}(u)$ and $\hat P^1_{f_1,f_2}(u)$ in $\mk[u]$. The polynomial $\hat P^0$ is obtained from $P^0$ just by replacing the variable $f\in \mF_0$ with $u_f$, and the polynomial $\hat P^1$ is obtained from $P^1$ by replacing the variable $f\in \mF_0$ with $u_f$ and the variable $f \in \mF_1$ with $F_f(u)$.  The straightening relation is then defined as
$$
\hat R_{f_1,f_2}(u)\doteq u_{f_1}\,u_{f_\rev{2}} - \hat P^0_{f_1,f_2}(u)- \hat P^1_{f_1,f_2}(u).
$$ 

\begin{teo}\label{teo:generalResult} If \hypOne,\ \hypTwo,...,\hypSix\ hold then 
\begin{enumerate}[\noindent i)]
\item for $f_1,f_2 \in \mF_0$ not $\tabr$ comparable the relations $\hat R_{f_1,f_2}(u)$ is a straightening relation w.r.t.\ $\newMonomialOrder$.
\item $(\mG_0,\tabr)$ is a SMT for $A$.
\end{enumerate}
\end{teo}

\begin{proof}
Let $f_1\in \mF_0(\gre_i)$ and $f_2\in \mF_0(\gre_j)$. The relations $R_{f_1,f_2}$ are $T_K$--homogeneous. By \hypSix\ this implies that $\hat P^ 0$ is homogeneous w.r.t.\ $\grado$ with degree equal to $\grado(\gre_i+\gre_j) $ and that $\grado(\hat P ^1) < \grado(\gre_i+\gre_j)$.  This immediately implies that $\hat R_{f_1,f_2}$ is a straightening relations w.r.t.\ $\newMonomialOrder$.

Hence the monomials $g_m$ with $m \in \mS\mM_0$ generate $A$ as a $\mk$--vector space. We want to prove that they are a basis. Notice that $\{g_m: m \in \mM_0 \mand \grado(u_m) \leq n \}$ spans $A_{\grado \leq n}$ for every $n \in \mN$ and, using that \rev{the} $\hat R$\rev{'s} are straightening relations, it follows that $\{g_m: m \in \mS\mM_0 \mand \grado(u_m) \leq n \}$ spans $A_{\grado \leq n}$ for every $n \in \mN$. Similarly, since $\calR$ is irreducible, $(\grG_\calR)_{\grado \leq n }$ is spanned by $\{m \in \mS\mM_0 \mand \grado (u_m)\leq n\}$ and, moreover, these are linearly independent elements. Now the assertion follows since $(\grG_\calR)_{\grado \leq n }$ and $A_{\grado \leq n}$ are finite dimensional and, being isomorphic as $G$--modules, they have the same dimension. 
\end{proof}

\rev{
The above proof that the straightening relations hold is, even if stated in a slightly more general setting, the same we gave in our previous paper \cite{CLM} (with Peter Littelmann). As seen above, it implies the existence of the SMT. We want to remark that in \cite{CLM} the proof of the existence of a SMT is independent from the existence of straightening relations, but it is wrong, in the sense that we tacitly assumed that the map $s:\Gamma_\calR\longrightarrow\Gamma_\calF$ is $G$--equivariant. This is not always the case, as pointed out in Remark \ref{remark_controesempio} at the beginning of Subsection \ref{ssec:compatibilita}.
}


\begin{cor}\label{cor:generalDegeneration}
There exists a two steps $\mk^* \times G$--equivariant flat degeneration of $A$ to $\grG_\calR$.
\end{cor}

\begin{proof}
\rev{Recall that} $\calS \rev{=} \overline {B_K \calR}$ \rev{is} the smallest Schubert variety containing $\calR$ and $\grG_\calS=\oplus_{n\geq0}\Gamma(\mathcal{S},\mathcal{L})$ \rev{is} its ring of coordinates. Since $\calR$ is a ridge Schubert variety we know that $\grG_\calR=\grG_\calS/(f_0)$. The function $\grf(f_0)$ is $G$--invariant (since $G\subset Q$) so it is a constant function; let it be equal to $c\in\mk$. Let $B\doteq\grG_\calS/(f_0-c)$. It is clear that the ring $B$ can be deformed to $\grG_\Rich$ in a flat and $\mk^*\times G$--equivariant way. We exhibit now a $\mk^* \times G$--equivariant deformation of $A$ to $B$. If $f \in \mF_1$ let $v_f$ be a new variable and if $f \in Z^*_n$ set either $\chi_f \doteq n - \grado(\grl_f)$ if $\grl_f \in \Omega_A^+ $ and $\chi_f\doteq0$ otherwise (for the definition of $\grl_f$ see the paragraph before Remark \ref{oss:calcolo}). By \hypSix, $\chi_f>0 $ for all $f\in \mF_1\senza\{f_0\}$.

Consider now the polynomial ring $\mk[u,v,t]$ in the variables $u_f$ with $f\in \mF_0$, $v_f$ with $f \in \mF_1$ and $t$. For $f_1,f_2\in \mF_0$ not $\tabr$ comparable let $R'_{f_1,f_2}(u,v)$ be obtained from $R_{f_1,f_2}$ by replacing $f$ with either $u_f$ if $f\in \mF_0$ or with $v_f$ if otherwise $f\in \mF_1$. Consider the quotient $D$ of $\mk[u,v,t]$ by the ideal generated by the polynomials $R'_{f_1,f_2}(u,v)$ for $f_1,f_2\in \mF_0$ not $\tabr$ comparable and by the polynomials $v_f - t^{\chi_f}F_f(u)$ for $f \in \mF_1$. Notice that on $D$ there is an action of $G$, indeed the ideal generated by the polynomials $R'$ is $G$--invariant since $\calR$ is $G$--stable and the ideal generated by relations $v_f - t^{\chi_f}F_f(u)$ is invariant since $F_f$ is induced by the $G$--equivariant projection $\sfS^{a_1}(V^*_{\gre_1})\otimes \cdots \otimes \sfS^{a_\ell}(V^*_{\gre_\ell}) \lra V_{\grl_f}^*$. On $D$ there is also an action of $\mk^*$ defined by $z\cdot u_f \doteq \eta(z)\, u_f $, $z\cdot v_f \doteq z^n\, v_f $ if $f \in Z^*_n\rev{\supseteq Z_\eta^*}$ and $z \cdot t \doteq z\, t$. This action commutes with that of $G$.

Finally for $a \in \rev{\mk}$ let $D_a \doteq D/(\rev{t-a})$ and notice that $D_0 \isocan B$ and $D_\rev{1}\isocan A$ by the previous \rev{Theorem}.

The flatness of the deformation follows by \hypTwo.
\end{proof}

\section{The case of model varieties}\label{sec:modello}

In this section we explain how to construct $K$ and $\calR$ for the model varieties of $\SL(n)$, $\SO(2n+1)$ and $\Sp(2n)$. The construction we are going to see is a generalization of the one we gave in \cite{CLM} for symmetric varieties with restricted root system of types $\sfA$, $\sfB$, $\sfC$ or $\sfB\sfC$. This approach has some general pattern that we want to make evident here and we hope that, with few variations, it could be applied also in other cases.

Let us recall our construction for symmetric varieties whose restricted root systems is of type $\sfA$, $\sfB$, $\sfB\sfC$ or $\sfC$, in the context of the general approach of this paper. In \cite{CLM} we used the following two main properties for the construction of $K$ and $\calR$. If we properly order the weights $\gre_1,\dots,\gre_\ell$ then
\begin{enumerate}
\item[\symOne:](Lemma~41 in \cite{CLM}) For $1\leq i\leq \ell-1$ there exists $u_i\in W$ such that $u_i\gre_i=\gre_{i+1}-\gre_1$. In particular\rev{, as in the proof of Lemma~41 in \cite{CLM},} $V_{\gre_{i+1}}$ appears with multiplicity one in $V_{\gre_i}\otimes V_{\gre_1}$ ;
\item[\symTwo:](Corollary~20 in \cite{CLM}) $V^*_{\gre_{i+1}}\subset V^*_{\gre_1}\cdot V^*_{\gre_i}$ in the coordinate ring $A$ of the symmetric variety, for $i=1,\dots,\ell-1$.
\end{enumerate}

\rev{
The first property allows to define a ridge Richardson variety with the required properties by setting $w_0\doteq\id$, $w_1\doteq s_{\gra_0}$ and $w_{i+1}\doteq s_{\gra_0} u_iw_i$ for $i=1,2,\ldots,\ell-1$.

By \symTwo\ all ``fundamental'' spherical representations can be iteratively constructed by taking the tensor product with the representation $V_{\gre_1}$. This is the key point which allows us to construct the Dynkin diagram of the group $K$ by adding a single node (associated to the root $\gra_0$) to the diagram of $G$. As we will see later, the analogous property for model varieties is more complicated (see property \modTwo\ in Lemma \ref{lem:Model123}) and this will force us to add two nodes in order to have the Dynkin diagram of the right group $K$.
}

We want now to briefly see how the various hypotheses \hypOne,...,\hypSix\ are proved in \cite{CLM}. \hypOne\ follows by the description of spherical weights in Section~1 in \cite{CLM}. The hypotheses \hypTwo, \hypThree\ and \hypFour\ follow by Corollary~40, Lemma~43 and Theorem~42 \rev{(see also the proof of Theorem~44)} respectively in \cite{CLM}. For \hypFive\ one may look at the remark before Proposition~35 at page~322 in \cite{CLM} since, with notation as in \hypFive, $\lambda+\mu-\nu$ is a rational \rev{non negative} linear combination of simple roots of the restricted root system. Finally \hypSix\ follows by Corollary~37 in \cite{CLM}.

In the case of model varieties property \symOne\ holds without any change while property \symTwo\ is not true any more. Before giving the details of the construction of $K$ and $\calR$ for model varieties we recall some properties of these varieties.

\subsection{Model varieties} \label{ssec:modello}

A homogeneous model variety for a semisimple group $\bar G$ is an homogeneous quasi affine variety whose coordinate ring is the sum of all irreducible representations of $\bar G$ with multiplicity one. These varieties were studied by Bernstein, Gelfand, Gelfand and Zelevinsky in \cite{Gel1,Gel2,Gel3}. In particular for a homogeneous model variety $\bar G/\bar H$ we have $\Omega_A^+=\Lambda_{\bar G}^+$.

As proved in \cite{Luna}, for a simple group $\bar G$ of rank $r$ there exist $2^r$ or $2^{r-1}$ (depending on the group $G$ see \cite{Luna}, page~293) homogeneous model varieties up to isomorphism and there exists exactly one model variety such that all the other \rev{ones} are degenerations of this variety. We want to describe a Pl\"uckerization of the coordinate ring of this variety in the cases $\bar G=\SL(\ell+1)$, $\bar G=\SO(2\ell+1)$ and $\bar G=\Sp(2\ell)$. We denote by $G$ the universal cover of $\bar G$ and by $H$ the inverse image of $\bar H$ in $G$ so that $X=G/H=\bar G/\bar H$. If we order the Dynkin diagram of $G$ as in Bourbaki \rev{\cite{bourbaki}}, we can choose a basis of $\Omega_A^+$ as follows: $\gre_i\doteq\omega_i$ for all $i$ but $i=\ell$ and $\rev{\bar G}=\SO(2\ell +1)$ where we set $\gre_\ell\doteq 2\omega_\ell$. In particular \hypOne\ is satisfied.

We denote by $h_i$ an $H$--invariant non\rev{-}zero element of $V_{\gre_i}$. We now describe $\bar H$, and $h_i$ case by case.

\subsubsection{$\SL(2n+1)$} Here $\ell=2n$ is even. Let $W$ be an $n$--dimensional vector space and set $V\doteq W\oplus W^*\oplus \mC \,v$. Let $a_W$ be the standard symplectic form on $W\oplus W^*$ and extend it trivially to $V$. In this case $G=\SL(V)$, $H=\{g\in G\st g(v)=v$ and $g\cdot a_W=a_W\}$.

We have $h_1 =v$, $h_2= a_W\in \rev{\bigwedge}^2 V$ , $h_{2i}=h_2^{\wedge i}$ for $i=1,2,\ldots,n$, and $h_{2i+1}=h_1\wedge h_{2i}$ for $i=1,2,\ldots,n-1$.

Notice also that in this case the model variety $G/H$ is isomorphic to the symmetric variety $\SL(2n+2)/\Sp(2n+2)$.

\subsubsection{$\SL(2n)$} Here $\ell=2n-1$ is odd. Let $V$ be a $2n$--dimensional vector space and fix a non degenerate symplectic form $a_V \in \rev{\bigwedge}^2 V$ and a non\rev{-}zero vector $v\in V$. Then $H=\{g\in\SL(2n): g v=v$ and $g \cdot a_V= a_V\}$.

We have $h_1 =v$, $h_2=a_V\in \rev{\bigwedge}^2 V$, $h_{2i}=h_2^{\wedge i}$ for $i=1,2,\ldots,n-1$, and $h_{2i+1}=h_1\wedge h_{2i}$ for $i=1,2,\ldots,n-1$.

\subsubsection{$\SO(2\ell+1)$} Let $W$ be an $\ell$--dimensional vector space and set $V\doteq W\oplus W^*\oplus \mC \,v$. On $V$ define a non degenerate bilinear symmetric form $b_V$ such that $W$ and $W^*$ are orthogonal to $v$ and such that on $W\oplus W^*$ the symmetric form is induced by the pairing between $W$ and $W^*$.  Let also $a_W$ be the standard non degenerate symplectic form on $W\oplus W^*$ and extend it trivially to $V$.  In this case we have $\bar G=\SO(V,b_V)$ and $\bar H=\GL(W)$ acting naturally on $W\oplus W^*$ and trivially on $v$.

Notice that we have $V_{\gre_i}=\rev{\bigwedge}^i V$, for $i=1,2,\ldots,\ell$, and also $h_1 =v$, $h_2=a_W\in \rev{\bigwedge}^2 V$, $ h_{2i}=h_2^{\wedge i}$, for $i=1,2,\ldots,$ and $h_{2i+1}=h_1\wedge h_{2i}$ for $i=1,2,\ldots$

\subsubsection{$\Sp(2\ell)$} If $\ell$ is even let $W_1$ and $W_2$ be two $\ell$--dimensional vector spaces while, if $\ell$ is odd, let $W_1$ and $W_2$ be two vector spaces of dimension $\ell+1$ and $\ell-1$, respectively. Fix two symplectic forms $a_{W_1}$ and $a_{W_2}$ on $W_1$ and $W_2$; let $V\doteq W_1\oplus W_2$, $a_V \doteq a_{W_1} + a_{W_2}$ and $G\doteq\Sp(V,a_V)$. Choose a non\rev{-}zero vector $v$ in $W_1$ and define $H_1\doteq\{g\in\Sp(a_{W_1}): g(v)=v\}$ and $H_2\doteq\Sp(a_{W_2})$. Then $H=H_1\times H_2$.

Notice that $\rev{\bigwedge}^i V \isocan a_V \wedge \rev{\bigwedge}^{i-\rev{2}} V \oplus V_{\gre_i}$, for $i=1,2,\ldots,\ell-1$; we denote by $\pi_i:\rev{\bigwedge}^i\rev{V}\lra V_{\gre_i}$ the projection onto the second factor.

We have $h_1 =v$, $h_2=\pi_2(a_{W_1})$, $h_{2i}=\pi_{2i}(a_{W_1}^{\wedge i})$, for $i=1,2,\ldots$, and $h_{2i+1}=\pi_{2i+1}(h_1 \wedge a_{W_1}^{\wedge i})$ for $i=1,2,\ldots$

\bigskip

We have the following properties similar to \symOne\ and \symTwo.

\begin{lem}\label{lem:Model123}
If $G/H$ is one of the above described model varieties then
\begin{enumerate}
\item[\modOne:] $V_{\gre_i}\subset V_{\gre_1}^{\otimes i}$ with multiplicity one and $V_{\gre_i} \not\subset V_{\gre_1}^{\otimes j}$ for $j<i$;
\item[\modTwo:] $V^*_{\gre_{i+2}}\subset V^*_{\gre_2}\cdot V^*_{\gre_i}$ and $V^*_{\gre_{2i+1}}\subset V^*_{\gre_1}\cdot V^*_{\gre_{2i}}$ in the coordinate ring $A$;
\item[\modThree:] for all $i$ there exists an element $u_i$ in the Weyl group $W$ such that $u_i \gre_i=\gre_{i+1}-\gre_1$\rev{; m}ore precisely we may choose $u_i=s_1s_2\cdots s_i$.
\end{enumerate}
\end{lem}

\begin{proof}
Properties \modOne\ and \modThree\ are a straightforward computation. 

In order to prove property \modTwo\ recall that if $\grl,\mu,\nu \in \Omega^+_A$ and $p:V_\lambda \otimes V_\mu \lra W$ is the projection on the \rev{isotypical} component of type $V_\nu$, then the condition $V^*_\nu \subset V^*_\lambda \cdot V_\mu^*$ in the coordinate ring of $G/H$ is equivalent to $p(V^H_\lambda\otimes V^H_\mu )\neq 0$.

Notice also that for $\SL(\ell+1)$ and $\SO(2\ell+1)$ and $\lambda=\gre_i$, $\mu=\gre_j$ and $\nu=\gre_{i+j}$ with $i+j\leq \ell$, the projection $p$ is given by $p(x\otimes y)= x\wedge y$; further for $\Sp(2\ell)$ and $\lambda=\gre_i$, $\mu=\gre_j$ and $\nu=\gre_{i+j}$ with $i+j\leq \ell$ the projection $p$ is given by $p(\pi_i(x)\otimes \pi_j(y))= \pi_{i+j}(x\wedge y)$. Hence property \modTwo\ follows from the description of the elements $h_i$ given above.
\end{proof}

\rev{
Property \modTwo\ should be seen as the analogue of the property \symTwo\ for symmetric varieties. Both properties allow us to order the generators of the monoid of spherical weights in a suitable way.  We notice also that properties \symTwo\ and \modTwo\ are related to the projective normality of wonderful varieties. Indeed in \cite{CLM}, the property \symTwo\ is proved using the results in \cite{CM} about the projective normality of complete symmetric varieties. Projective normality for model varieties has been studied in \cite{BGM}. In this case, projective normality fails for the model varieties of $\sf{SO}(2r+1)$; however, here we need only a weaker property.
}

\subsection{Construction of $K$}

Let $\Delta\rev{\doteq}\{\gra_1,\dots,\gra_\ell\}$ be the set of simple roots of $\gog$ numbered as in Bourbaki \rev{\cite{bourbaki}}. For the construction of $K$ we consider the Dynkin diagram of $G$ and we add two nodes so that $\Delta_0=\{\gra_0,\grb_0\}$. This choice is suggested by the fact that, by property \modTwo, $h_1$ and $h_2$ generate all the $H$--invariants. We connect both nodes to the node corresponding to the simple root $\gra_1$. (However\rev{,} in the next Section we see an example where we make a different choice.) Notice that $K$ is symmetrizable.

Let $L$ be the Levi whose simple roots are given by $\Delta$ and let $C$ be its center, as in Section \ref{sec:pluckerization}. For an element $\eta \in \mZ[\Delta_0]$ we can consider the associated eigenspace $\goK_\eta$ in $\goK$ for the action of $C$ as we have explained in Section \ref{sec:pluckerization} for the module $Z^*$. We define $\goK^-$ as the sum of the eigenspaces $\goK_\eta$ for $\eta \in -\mN[\Delta_0]$.

The inclusion $\Delta\subset\Delta_K$ determines an inclusion of $\Lambda$ into $\Lambda_K$. However\rev{,} this inclusion does not send fundamental weights to fundamental weights. We denote by $\omega_1,\dots,\omega_\ell \in \Lambda$ the fundamental weights of $\gog$ w.r.t.\ $\Delta$ and by $\tom_{\gra_0},\tom_{\grb_0},\tom_1,\dots,\tom_\ell \in \Lambda_K$ a choice of fundamental weights for $\Delta_K$ as explained in Section \ref{sec:simboli}. If $\grl=\sum a_i \omega_i$ is an element of $\Lambda$ we denote by $\tilde \lambda \in \Lambda_K$ the element $\sum_i a_i\tom_i$.

\begin{lem}\label{lem:K}\hfill

\begin{enumerate}[\indent i)]
\item $\goK_0 = \gog + \got_K$;
\item $\goK^-$ is generated by $\goK_{-\gra_0}$ and $\goK_{-\grb_0}$;
\item $\goK_{-\eta}\isocan \goK_\eta^*$ as a $G$--module.
\item $\goK_{-\gra_0} \isocan \goK_{-\grb_0}\isocan V_{\gre_1}$ as a $G$--module;
\item the $G$--module $V_{\gre_2}$ appears in $\goK_{-\gra_0-\grb_0}$ with multiplicity one.
\end{enumerate}
\end{lem}

\begin{proof}
$i)$ and $ii)$ are trivial. The $T_K$\rev{--}character of $\goK_\eta$ is obtained from the character of $\goK_{-\eta}$ by composition with the involution $t\mapsto t^{-1}$, so $iii)$ follows.

We prove $iv)$ in the case of $\goK_{-\gra_0}$, the proof for $\goK_{-\grb_0}$ is analogous. Let $\sff_{\gra_0},\sff_{\grb_0},\sff_1,\dots,\sff_\ell$ be root vectors of weight $-\gra_0,-\grb_0,-\gra_1,\dots,-\gra_\ell$. We prove that $\sff_{\gra_0}$ generates $\goK_{-\gra_0}$ as a $G$--module. Indeed $\goK_{-\gra_0}$ is generated by vectors $y$ of the form $[\sff_{i_1},\dots,[\sff_{i_a},[\sff_{\gra_0},[\sff_{j_1},\dots,[\sff_{j_{b-1}},\sff_{j_b}]\cdots]$ where $i_h,j_k\in\{1,\dots \ell\}$. In particular if we set $x\doteq-[\sff_{j_1},\dots,[\sff_{j_{b-1}},\sff_{j_b}]\cdots]$ we have $y=[\sff_{i_1},\dots,[\sff_{i_a},[x,\sff_{\gra_0}]\cdots]$ which is in the $G$--module generated by $\sff_{\gra_0}$. Finally notice that $\sff_{\gra_0}$ is a highest weight vector of weight $\gre_1$ for the action of $G$.

We now prove $v)$. The weight $\grg_0\doteq\gra_0+\grb_0+\gra_1\equiv \tom_{\gra_0}+\tom_{\grb_0}-\tilde\gre_2$ is a real root of $\goK$. Let $y\in \goK$ be a non\rev{-}zero element of weight $-\grg_0$. In particular $y\in \goK_{-\gra_0-\grb_0}$ is a highest weight vector of weight $\gre_2$ for the action of $G$, hence $V_{\gre_2}$ appears in $\goK_{-\gra_0-\grb_0}$ with non\rev{-}zero multiplicity. Moreover all roots $\grg\in \Phi_K$ which appear with non\rev{-}zero multiplicity in the space $\goK_{-\gra_0-\grb_0}$ are of the form $\gamma=-\grg_0-\eta$ with $\eta \in \mN[\Delta]$. Hence all other possible $G$--highest weight vectors in $\goK_{-\gra_0-\grb_0}$ have a weight which is strictly smaller than $\gre_2$ w.r.t.\ dominant order in $\Lambda$. In particular  $V_{\gre_2}$ appears with multiplicity one in $\goK_{-\gra_0-\grb_0}$.
\end{proof}

\subsection{Construction of $\calR$}

Let $Z$ be the irreducible highest weight module of highest weight $\tom_{\gra_0}$. Let $Q$ be the maximal parabolic of $K$ corresponding to the root $\gra_0$. Let $\calF\doteq K/Q$ and let $\calL$ be the line bundle on $\calF$ such that $\Gamma(\calF,\calL)$ is isomorphic to the dual $Z^*$ of $Z$.

Define $w_0\doteq\id$, $w_1\doteq s_{\gra_0}$, $w_{2i}\doteq s_{\grb_0} u_{2i-1}w_{2i-1}$ and $w_{2i+1}\doteq s_{\gra_0} u_{2i}w_{2i}$ for $i=1,2,\ldots$ where $u_1,u_2,\ldots$ are the elements which appear in the above property \modThree. Notice $\rev{w_{2i-1} (\tom_{\gra_0}) \equiv \tilde \gre_{2i-1}-\tom_{\gra_0}}$ and $\rev{w_{2i} (\tom_{\gra_0}) \equiv \tilde \gre_{2i}-\tom_{\grb_0}}$ for $i=1,2,\ldots$. \rev{Further the length requirements $\length(w_i)=1+\length(u_i)+\length(w_{i-1})$, for $i=1,2,\ldots$ of Subsection \ref{ssec:compatibilita} may be easily checked by induction. Indeed, by applying the simple symmetries in $u_i = s_1 s_2\cdots s_i$ to $w_{i-1}(\tom_{\gra_0})$, we obtain a strictly ordered sequence of $i$ weights (with respect to the dominant order).} In particular we can associate \rev{with} $w_0,w_1,\dots,w_\rev{\ell}$ a ridge Richardson variety $\calR$ such that \hypTwo\ is satisfied.

Since $\grado(\gre_i) = i$ we have $\grado(\gra_i)=0$ for $i=1,\dots,\ell-1$ and $\grado(\gra_\ell)>0$.  In particular \hypFive\ is satisfied.

\begin{lem}\label{lem:IP6}\hfill
\begin{enumerate}[\indent i)]
\item \hypSix\ is satisfied;
\item if $V_\grl\subset \goK_{-\gra_0-\grb_0}$ with $\grl\in\Omega^+_A$ and $\grado(\grl)\geq 2$ then $\grl=\gre_2$;
\item if $V_\grl\subset \goK_{-a\gra_0-b\grb_0}$ with $\grl\in\Omega^+_A$ and $\grado(\grl)\geq a+b$ then $(a,b)\in\{(0,1),(1,0),(1,1)\}$.
\end{enumerate}
\end{lem}

\begin{proof}
Notice first that by Lemma \ref{lem:K} $ii)$, $iii)$ and $iv)$ we have a surjective $G$--equivariant homomorphism from $(V_{\gre_1}^*)^{\otimes n}$ to $Z^*_n$.  Hence if $V_\grl^*$ appears in $Z^*_n$ as \rev{a $G$--submodule} then $\grado(\grl)\leq n$.

Let now $M\isocan V^*_\grl$ be a submodule of $Z^*_n$ with $\grado(\grl)=n$. Consider the Levi subgroup $L_K$ whose simple roots are $\gra_0,\grb_0,\gra_1,\dots,\gra_{\ell-1}$ and let $U$ be the $L_K$--submodule of $Z^*$ generated by the lowest weight vector $f_0$. This is an irreducible representation of $L_K$ of lowest weight $-\tom_{\gra_0}$. Consider also the Levi subgroup $L_G$ of $G$ with simple roots $\gra_1,\dots,\gra_{\ell-1}$.

As above, by Lemma \ref{lem:K} $ii)$, $iii)$ and $iv)$, we know that $V^*_\grl$ appears with non\rev{-}zero multiplicity in $(V_{\gre_1}^*)^{\otimes n}$, hence $\grl=n\gre_1 - \grg$ with $\grg\in \mN[\Delta]$. Furthermore since $\grado(\grl)=\grado(n\gre_1)$ and $\grado(\gra_\ell)>0$ we have $\grg\in\mN[\gra_1,\dots,\gra_{\ell-1}]$. So there exists a vector in $Z_n^*$ of weight $-\mu\in\Lambda_K$ such that $\mu\equiv \tilde \lambda$, moreover $\tom_{\alpha_0}-\mu=\gamma+a\alpha_0+b\beta_0$ with $a+b=n$. Since $\alpha_\ell$ does not appear in $\tom_{\alpha_0}-\mu$, $-\mu$ is a weight of a non\rev{-}zero vector of the module $U$ and it is non positive against $\alpha_1\cech, \alpha_2\cech, \dots, \alpha_{\ell-1}\cech$.

Finally notice that the semisimple part of $L_K$ is a group of type $\sfD_{\ell+1}$ and $U$ is a spin module. This is a minuscule module, so $\mu$ must be in the orbit of $\tom_{\alpha_0}$ for the Weyl group of $L_K$. Since $\mu$ is dominant with respect to $\alpha_1\cech, \alpha_2\cech, \dots, \alpha_{\ell-1}\cech$, a simple computation shows that it must be equivalent to one of the following weights w.r.t.\ the relation $\equiv$
$$
\tom_{\alpha_0}, \tom_1-\tom_{\alpha_0}, \tom_2-\tom_{\beta_0}, \tom_3-\tom_{\alpha_0},\tom_4-\tom_{\beta_0},\ldots
$$ 
The condition $a+b=n$ gives $n\leq \ell$ and either $\lambda=0$ if $n=0$, or $\lambda=\omega_n$ if $0<n\leq \ell$. Further notice that $\mu=w_n\tom_{\alpha_0}$, so the module $M$ is the $G$--module spanned by $\mF(\gre_n)$. This proves $i)$.

We now prove $ii)$ and $iii)$. Let $V_\grl\subset \goK_{-a\gra_0-b\grb_0}$ with $\grado(\grl)=a+b$.  Then there exists a root $\grd\in\Phi_K$ of the form $a\gra_0+b\grb_0+\gamma$ with $\grg\in \mN[\Delta]$ and $\gamma = -\grl + (a+b)\gre_1$. From $\grado(\grl)=a+b$ we see that the root $\gra_\ell$ does not appear in $\gamma$. In particular $\grd$ is a root of a root subsystem of type $\sfD$ (notice however that this system is not numbered in the usual way). We immediately deduce that $a\leq 1$ and $b\leq 1$. The case $a=0$ or $b=0$ is setted by Lemma \ref{lem:K} point $iv)$. We now study the case $a=b=1$. We notice that the pairing of $\gamma$ with $\gra_2\cech,\dots,\gra_\ell\cech$ must be non positive. From the explicit description of a positive root in a root system of type $\sfD$ we immediately deduce that $\gamma=\gra_1$ and $\grl=\gre_2$.
\end{proof}

\subsection{Construction of $\grf$}

The strategy we use is a generalization of the one we adopted in \cite{CLM}. Let $x_0 \in \calF$ be the element fixed by $B_K$. By Lemma \ref{lem:K} there exists $z_1\in \goK_{-\gra_0}$ and $z_2 \in V_{\gre_2}\subset \goK_{-\gra_0-\grb_0}$ invariant under $H$. We define $z\rev{\doteq}z_1+z_2$. The idea is to consider the element $\exp(z)\cdot x_0$; since this element is $H$--invariant the action of $G$ determines an embedding of $G/H$ in $\calF$ and $\grf$ is defined as the pull--back with respect to this map.

In order to verify all the required hypotheses we need to check some properties. In particular in Kumar's construction of the group $K$ in \cite{Kumar}, the element $\exp(z)$ does not \rev{exist}, and we have to go through a technical detour which is completely analogous to the case of symmetric varieties. For this reason we illustrate here the main steps and we refer to \cite{CLM} Section 5.1 and 5.2 for the details.

As above let $Z$ be the irreducible highest weight module of weight $\tom_0$. Recall that $Z$ is the restricted dual of $Z^*=\Gamma(\calF,\calL)$, and set $Z_{-n} \doteq Z_n^*$.  Let $v_0$ be a highest weight vector of $Z$ and let $J_{n}\doteq\bigoplus_{m\leq n}Z_m$ and $f_0$ a lowest weight vector in $Z^*$ such that $\langle v_0,f_0\rangle=1$. \rev{Since the} subspace $J_m\subset Z$ \rev{has} finite codimension, the element $x_m\doteq e^z \bar v_0 \in Z/J_m$ is well defined and is $H$--invariant.  We can define a map $\mi_m:G/H\lra \mP(Z/J_m)$ by $\mi_m(gH)\doteq g\,x_m$.  Since $e^z\cdot \bar v_0$ is $H$--invariant, $\mi_m^*(\calO_{\mP(Z/J_m)}(1))$ is the trivial bundle on $G/H$. Hence it defines a map $\mi^*_m$ from the annhilator $J_m^0$ of $J_m$ in $Z^*$ to $\mk[G/H]$. We can normalize this map by requiring that $\mi^*_m(f_0)$ is the constant function $1$. We notice that if $m<n\leq 0$ then $\mi^*_m\ristretto_{J_n^0}=\mi_n^*$ and this allows to define a map $\mi^*:Z^*\lra \mk[G/H]$. Finally, as proved in Lemma~43 \cite{CLM}, we see that the symmetric power of $\mi^*$ from $\sfS(Z^*)$ to $\mk[G/H]$ factor through a map $\grf:\Gamma_{\calF}\lra \mk[G/H]$. Hence \hypThree\ is satisfied.

\subsection{Verification of \hypFour}

By the construction of $\calR$ and by \hypSix, for $i=1,\dots,\ell$, there exists a unique $G$--submodule of $Z_{-i}$ isomorphic to $V_{\gre_i}$. We denote this module by $M_{-i}$ and we denote by $q_i:Z_{-i}\lra M_{-i}$ the projection onto this factor. 

\begin{lem}\label{lem:zv}
Let $a\doteq\sum a_i$ and $b\doteq\sum b_i$\rev{,} then
\begin{enumerate}[\indent i)]
\item $q_{n}(z_1^{a_1}z_2^{b_1}\cdots z_1^{a_m}z_2^{b_m}\cdot v_0)=q_{n}(z_1^az_2^b\cdot v_0)$;
\item $q_{2i}(z_1^a\,z_2^b\cdot v_0) = 0$ for $a\neq 0$ or $b\neq i$; 
\item $q_{2i}(z_2^i\cdot v_0) \neq 0$;
\item $q_{2i+1}(z_1^a\,z_2^b\cdot v_0) = 0$ for $a\neq 1$ or $b\neq i$; 
\item $q_{2i+1}(z_1\,z_2^i\cdot v_0) \neq 0$. 
\end{enumerate}
\end{lem}

\begin{proof}
We decompose $Z$ according to the action of $C$ as we did for the Lie algebra $\goK$: so $Z_{-a\gra_0-b\grb_0}$ is the span of the $T_K$--eigenvectors of weight of the form $\tom_0-a\gra_0-b\grb_0-\grg$ with $\grg \in \mN[\Delta]$. We have $Z_{-n}=\bigoplus_{a+b=n}Z_{-a\gra_0-b\grb_0}$.  Further $M_{-2i}\subset Z_{-i\gra_0-i\grb_0} \quad$ and $\quad M_{-2i-i}\subset Z_{-(i+1)\gra_0-i\grb_0}$
In particular $q_{2i}(z_1^{a_1}z_2^{b_1}\cdots z_1^{a_m}z_2^{b_m}\cdot v_0)=0$ if $(a,b)\neq (0,i)$ and  $q_{2i+1}(z_1^{a_1}z_2^{b_1}\cdots z_1^{a_m}z_2^{b_m}\cdot v_0)=0$ if $(a,b)\neq (1,i)$. This implies $ii)$ and $iv)$.

By the previous remark, $i)$ is proved if we show that
$$
q_{n}(z_1^{a_1}z_2^{b_1}\cdots [z_1,z_2]\cdots z_1^{a_m}z_2^{b_m}\cdot v_0)=0
$$
under the assumption $\sum a_i +2\sum b_i +3 =n$.  If this is not the case we would have a non\rev{-}zero, hence surjective, $G$--equivariant map 
$$
\goK_{-\gra_0}^{\otimes a_1}\otimes\goK_{-\gra_0-\grb_0}^{\otimes b_1}\otimes \cdots
\goK_{-2\gra_0-\grb_0}\otimes\cdots \goK_{-\gra_0-\grb_0}^{\otimes b_m} \lra M_{-n}
$$
given by
$$
x_1\otimes x_2 \cdots \otimes x_r \longmapsto q_n(x_1\cdot x_2 \cdots v_0).
$$
However by Lemma \ref{lem:IP6} $iii)$ and \hypFive, $V_{\gre_n}$ is not a submodule of $ \goK_{-\gra_0}^{\otimes a_1}\otimes\goK_{-\gra_0-\grb_0}^{\otimes b_1}\otimes \cdots \goK_{-2\gra_0-\grb_0}\otimes\cdots \goK_{-\gra_0-\grb_0}^{\otimes b_m}$ and we obtain a contradiction.

We now prove $iii)$ and $v)$. We order the weights in $\mN[\Delta]$ with a complete order $>_{\Phi}$ in such a way that $\gra_0+\grb_0>_\Phi\gra_0>_\Phi\grb_0$.  For any decreasing sequence of weights $\eta_1>_\Phi \eta_2 >_\Phi \cdots >_\Phi\eta_r$ we can consider the $G$--equivariant map
$$
\goK_{-\eta_1}\otimes \goK_{-\eta_2}\otimes \cdots \goK_{-\eta_r}\lra M_{-n} 
$$
given by
$$
x_1\otimes x_2 \cdots \otimes x_r \longmapsto q_n(x_1\cdot x_2 \cdots x_r\cdot v_0).
$$
By \hypFive, and Lemma \ref{lem:IP6} $iii)$ and by $\langle\grb_0\cech,\tom_{\gra_0}\rangle = 0$, this map is zero \rev{unless} $\eta_i\rev{\in} \{\gra_0,\gra_0+\grb_0\}$ for some $i$. Hence, using that $Z$ is generated by $v_0$, the maps 
$$
\goK_{-\gra_0-\grb_0}^{\otimes i}\otimes\goK_{-\gra_{\rev{0}}}  \lra M_{-2i-1}
\quad
\mand
\quad
\goK_{-\gra_0-\grb_0}^{\otimes i} \lra M_{-2i}
$$
defined as above are surjective. This implies, in turn, that the map
$$
\rev{\goK}_{-\gra_0-\grb_0} \otimes Z_{-n+2}\lra M_{-n} \; \text{ given by } \; x\otimes v \mapsto q_n(x\cdot v) 
$$
is surjective. Denote by $N$ the unique $G$--submodule of $\rev{\goK}_{-\gra_0-\grb_0}$ isomorphic to $V_{\gre_2}$ and let $r_n:N\otimes M_{-n+2}\lra M_{-n}$ be the restriction to $N\otimes M_{-n+2}$ of the last map above.  By \hypSix, Lemma \ref{lem:IP6} $iii)$ and \hypFive, $M_{-n}$ is in the image of $r_n$. Further by property \modOne\ we know that $V_{\gre_n}$ appears with multiplicity at most one in $V_{\gre_2}\otimes V_{\gre_{n-2}}$. Hence $r_n$ is the projection onto the \rev{isotypical} component of type $V_{\gre_n}$ of $V_{\gre_2}\otimes V_{\gre_{-n+2}}$ and it is uniquely determined up to a non\rev{-}zero scalar factor.

We now proceed by induction on $i$. For $i=1$ the statements are clear since $\langle\gra_0\cech,\tom_{\gra_0}\rangle\neq 0 $ and $\langle(\gra_0+\gra_1+\grb_0)\cech,\tom_{\gra_0}\rangle\neq 0 $. Let $n$ be either $2i$ in case $iii)$ or $2i+1$ in case $v)$. If $n>3$ we need to prove that $r_n(h_2\cdot h_{-n+2} )\neq 0$ where $h_i$ is a non\rev{-}zero $H$--invariant vector in $V_{\gre_i}$. Now as already noticed $r_n(h_2\cdot h_{-n+2} )\neq 0$ is equivalent to property \modTwo.
\end{proof}

We can finally verify \hypFour.

\begin{lem}
\hypFour\ is satisfied.
\end{lem}

\begin{proof}
Let $B_n$ be the span of $\mF_{\rev{0}}(\gre_n)$. As seen in Section \ref{ssec:compatibilita} $\rev{s}$ gives a $G$--equivariant isomorphism between $\Gamma(\calR,\calL)$ and $\bigoplus_{n=1}^\ell B_n$, moreover $B_n\isocan V^*_{\gre_n}$ as a $G$--module. Hence to prove that $\grf_1$ is an isomorphism it is enough to prove that $\mi^*(B_n)\neq 0$. This is equivalent to $q_i(e^z\cdot \bar v_0) \neq 0$ and this follows by the previous Lemma.
\end{proof}

\section{Another class of spherical varieties}\label{sec:another}

In this section we apply our method to another class of spherical varieties, listed as $(15)$ in the paper \cite{Bravi} page $656$. Let $2\leq p\leq n-2$, $V\doteq U\oplus U^*\oplus W\oplus\mC v$ with $U$ a vector space of dimension $p$, $W$ a vector space of dimension $2n-2p$. Let $B$ be a symmetric non degenerate bilinear form on $V$ such that $(U\oplus U^*)\oplus W\oplus \mC v$ is an orthogonal decomposition of $V$ and $B$ restricted to $U\oplus U^*$ is the natural symmetric bilinear form on this vector space.

For this example $\bar G=\SO(V,B)\simeq\SO(2n+1)$, further $\bar H$ is the subgroup of the elements $g\in \bar G$ such that $g(U)=U$ and $g(v)=v$. For this spherical variety $\Omega_A^+$ is the free monoid generated by the weights $\gre_1=\omega_1$, $\gre_2=\omega_p$ and $\gre_3=\omega_{p+1}$. The $\bar H$--invariants $h_i\in V_{\gre_i}$, $i=1,2,3$, are described as follows: $h_1=v$, $h_2$ is a non\rev{-}zero vector of $\Lambda^p(U)\subset\Lambda^p(V)$ and $h_3=h_1\wedge h_2$. The following Lemma can be easily checked.

\begin{lem} For the spherical variety $\bar G/\bar H$ we have:
\begin{enumerate}
\item[\sphOne:] $V_{\gre_2}\subset V_{\gre_1}\otimes V_{\omega_{p-1}}$ with multiplicity $1$ and $V_{\gre_3}\subset V_{\gre_2}\otimes V_{\gre_1}$ with multiplicity $1$;
\item[\sphTwo:] $V_{\gre_3}\subset V_{\gre_1}\cdot V_{\gre_2}$ in the coordinate ring of $\bar G/\bar H$;
\item[\sphThree:] $u_1=s_{p-1}s_{p-2}\cdots s_2s_1$, $u_2=s_1s_2\cdots s_{p-1}s_p$ are such that $u_1\gre_1=\gre_2-\omega_{p-1}$, $u_2\gre_2=\gre_3-\gre_1$.
\end{enumerate}
\end{lem}

We now describe $K$ and $\calR$. $K$ is the Kac--Moody group whose Dynkin diagram is obtained from the Dynkin diagram of $G$ by adding two nodes: a node $\alpha_0$ connected with $\alpha_1$ and a node $\beta_0$ connected to $\alpha_{p-1}$. Notice that also in this case $K$ is symmetrizable.  With the same notation of the previous \rev{S}ection we have the following results analogue of Lemma \ref{lem:K}:
\begin{enumerate}
\item[i)] $\goK_{-\gra_0}\isocan V_{\gre_1}$ as a $\gog$--module;
\item[ii)] $\goK_{-\grb_0}\isocan V_{\omega_{p-1}}$  as a $\gog$--module;
\item[iii)] $\goK_{-\gra_0-\grb_0}$ as a $\gog$--module and it contains $V_{\gre_2}$ with multiplicity one.
\end{enumerate}

Let $Q$ be the standard maximal parabolic subgroup of $K$ relative to the root $\alpha_0$, $\calF \rev{\doteq} K/Q$. The ridge Richardson variety $\calR$ is defined by $w_0 \rev{\doteq} \id$, $w_1 \rev{\doteq} s_{\alpha_0}$, $w_2 \rev{\doteq} s_{\beta_0} u_1 w_1$, $w_3 \rev{\doteq} s_{\alpha_0}u_2w_2$.

The construction of $\grf$ and the verification of the hypotheses \hypOne,\hypTwo,...,\hypSix\ are completely similar to the case of model varieties.

We have included here this example since we believe that our theory may be applied to other classes of spherical varieties in a similar way.

\bibliographystyle{plain}

\def\cprime{$'$}

\end{document}